\documentclass[letter, 10pt]{article}
\usepackage{amssymb}
\usepackage{amsthm}
\usepackage{amsmath}
\usepackage{mathrsfs}
\usepackage{verbatim}
\usepackage{enumerate}

\newcommand{\R}{\mathbb{R}}
\usepackage{cite}
\usepackage{hyperref}
\usepackage{changes}
\usepackage{color}

\numberwithin{equation}{section}

\newtheorem{theorem}{Theorem}[section]
\newtheorem{theorem*}{Theorem}

\newtheorem{corollary}[theorem]{Corollary}
\newtheorem{remark}[theorem]{Remark}

\newtheorem{definition}[theorem]{Definition}

\newtheorem{example}[theorem]{Example}

\newcommand\argmin{{\rm argmin}}
\newcommand\vol{{\rm vol}}

\newcommand\supp{\mathop{\rm supp}}

\begin{document}
\title{A canonical barycenter via Wasserstein regularization \footnote{Y.-H.K. is supported in part by 
Natural Sciences and Engineering
Research Council of Canada (NSERC) Discovery Grants 371642-09 and 2014-05448 as well as Alfred P. Sloan research fellowship 2012-2016.  B.P. is pleased to acknowledge the support of a University of Alberta start-up grant and National Sciences and Engineering Research Council of Canada Discovery Grant number 412779-2012. Part of this research was done while both authors were visiting  the Fields Institute, Toronto, ON., in the Fall 2014, in the thematic program `Calculus of Variations', and while B.P was visiting the Pacific Institute for the Mathematical Sciences (PIMS) in February 2017. We thank both institutes for their hospitality and support. 
}}

\author{Young-Heon Kim\footnote{Department of Mathematics, University of British Columbia, Vancouver BC Canada V6T 1Z2 \ \ yhkim@math.ubc.ca} and Brendan Pass\footnote{Department of Mathematical and Statistical Sciences, 632 CAB, University of Alberta, Edmonton, Alberta, Canada, T6G 2G1 \ \ pass@ualberta.ca}}
\maketitle
\begin{abstract}
 We introduce a weak notion of barycenter of a probability measure $\mu$ on a metric measure space $(X, d, {\bf m})$, with the metric $d$ and reference measure ${\bf m}$. Under the assumption that optimal transport plans are given by mappings, we prove that our barycenter $B(\mu)$  is well defined; it is a probability measure on $X$ supported on the set of the usual metric barycenter points of the given measure $\mu$. The definition uses the canonical embedding of the metric space $X$ into its Wasserstein space $P(X)$, pushing a given measure $\mu$ forward to a measure on $P(X)$. We then regularize the measure by the Wasserstein distance to the reference measure ${\bf m}$, and obtain a uniquely defined measure on $X$ supported on the barycentric points of $\mu$. We investigate various properties of $B(\mu)$.

\end{abstract}


\section{Introduction}\label{sec: intro}
In this paper, we propose a canonical weak notion of barycenter of a probability measure on a compact metric measure space $(X, d, {\bf m})$, where we assume that ${\bf m}$ is a probability measure. 

For a separable compact metric space $X$, we let $P(X)$ denote the space of Borel probability measures equipped with the weak-* topology.   Given a measure $\mu \in P(X)$, a barycenter (or Frechet mean) on $\mu$ is  a minimizer of the average squared distance to points in the support of $\mu$; that is, an element of 
\begin{align}\label{eq:barycenterset}
b(\mu) =\argmin \Big( y \mapsto \int _X d^2(x,y) d\mu(x)\Big ).
\end{align}
These metric barycenters are natural generalization of centers of mass of distributions on  Euclidean space. As the function $y \mapsto \int _X d^2(x,y) d\mu(x)$ is in general highly non-convex, except under strong additional criteria (for example, that $X$  a simply connected, non-positively curved space), the minimizer may not be unique; that is, $b(\mu)$ may not be a singleton.    Choosing a unique canonical minimizer in any reasonable sense is clearly impossible; for example, if $\mu$ is a uniform measure on the equator of the round sphere, the north and south pole are both barycenters and neither is any more natural than the other. 
To develop a canonical notion of barycenter, we utilize the geometry of $P(X)$ as follows.

 Via the natural isometry $x \mapsto \delta_x$, one may isometrically embed the  set $X$ into $P(X)$ equipped with Wasserstein metric 
\begin{equation}\label{eqn: wass distance}
W_2(\mu,\nu) := \inf\int_{X\times X} d^2(x,y)d\gamma(x,y)
\end{equation}
where the infimum is over the set of all probability measures $\gamma$ on $X \times X$ with marginals $\mu$ and $\nu$.  The minimization \eqref{eqn: wass distance} is the well known Monge-Kantorovich problem, reviewed for example, in books of Villani \cite{V,V2}, and more recently Santambrogio \cite{Santambrogio15p}; our particular interest here lies in the fact that the Wasserstein distance extends the underlying geometry on $X$ to $P(X)$.  It is also relevant here (and well known) that $P(X)$ with this metric inherits compactness from $X$, and that the Wasserstein metric topology coincides with the weak-* topology.

 One can then consider instead the barycenter of the image measure $(x \mapsto \delta_x)_\#\mu$ on the metric space $(P(X), W_2)$, which amounts to finding a minimizer of 
\begin{eqnarray}\label{eq:wass-barycenter}
\nu \mapsto \int_X W_2^2(\nu,\delta_x)d\mu(x)& =&\int_X \int_X d^2(y,x)d\nu(y)d\mu(x)\\\nonumber
& =&\int_X \int_X d^2(y,x)d\mu(x)d\nu(y).
\end{eqnarray}

The general study of barycenters in the Wasserstein space was initiated by Agueh-Carlier \cite{ac} when the underlying space $X =\mathbb{R}^n$ is Euclidean, and continued by the present authors to the setting where $X$ is a smooth Riemannian manifold \cite{KimPass2017}.  These represent a natural, non-linear way to interpolate between several (or infinitely many) probability measures, and have received a great deal of attention in recent years, due in part to important applications in image processing and statistics; see, for example, the work of Rabin et al. \cite{RabinPeyreDelonBernot2012} and Bigot and Klein \cite{BigotKlein2013} among others.
Clearly, if $x \in BC(\mu)$ is a barycenter of $\mu$, then $\delta_x$ is a barycenter of $(x \mapsto \delta_x)_\#\mu$, and so this minimization can be considered a relaxation of the barycenter problem on $X$.  Although the later problem is on the infinite dimensional space $P(X)$ rather than $X$, it has the advantage of being a linear minimization.  On the other hand, it certainly does not resolve the non-uniqueness issue; the minimizers of \eqref{eq:wass-barycenter} are exactly those measures which are supported on $b(\mu)$, while only the Dirac masses supported on the same set correspond to classical barycenters.

On the other hand, one can try to define a canonical, distinguished minimizer of \eqref{eq:wass-barycenter}.  For that purpose, throughout the paper  we assume that  $(X,d,{\bf {m}})$ satisfies a {\em regularity} condition (see Definition~\ref{def:regularity}); that is, 
 the minimizers of \eqref{eqn: wass distance} in the definition of $W_2({\bf m},\nu)$ are realized as  mappings pushing ${\bf m}$ forward to $\nu$. Examples of regular spaces $X$ include any smooth Riemannian manifold (with ${\bf m}$ absolutely continuous with respect to local coordinates) due to the Brenier-McCann theorem \cite{bren, m3}, and more general (singular) metric measure spaces including the 
$CD(K,N)$ space in the sense of Lott-Villani-Sturm \cite{Lott-Villani, sturm06, sturm06a}; see the work of Cavaletti and Huesmann \cite{Cavalletti20151367}, Gigli \cite{Gigli2012}, and 
 Gigli, Rajala, and Sturm \cite{Gigli2016}.  
 In this case, we show that one can pick the unique minimizer which is the most spread out, or the closest to the reference measure ${\bf m}$.  This is the approach adopted in this paper.

For $\epsilon >0$, consider the function on $P(X)$ defined by
\begin{align}\label{eq:F-epsilon}
F_\epsilon(\nu) :=\int_X \int_X d^2(y,x)d\mu(x)d\nu(y) +\epsilon W_2^2({\bf m},\nu).
\end{align}
It is not hard to establish  using regularity of $X$, that $F_\epsilon$ has a unique minimizer $\mu_\epsilon$. 
In our main theorem (Thereom~\ref{thm:B-mu-char}) we show 
 that as $\epsilon \rightarrow 0$, this minimizer converges weakly-* to a {\em unique} measure $\nu$ which minimizes $ W_2^2({\bf m} ,\cdot)$ among all minimizers of \eqref{eq:wass-barycenter}.   We call this distinguished minimizer the \emph{Wasserstein regularized barycenter of $\mu$}, and denote it by $B(\mu)$. This Wasserstein regularized barycenter is a weak notion of the metric barycenter;  moreover,  we verify that $B(\mu)$ is the unique minimizer of $\nu \mapsto W_2 ({\bf m}, \nu)$ among all $\nu \in P(X)$ with $\supp \nu \subset b(\mu)$.

The mapping $\mu \mapsto B(\mu)$ induces a subtle dynamical structure on $P(X)$, and we go on to establish some notable properties of it: the mapping reduces the variance (Corollary~\ref{cor: variance reduction}), has fixed points (including Dirac masses, but also others, see Remark \ref{rem: instability} and Example \ref{ex: circle}), has periodic orbits of period 2 (Examples \ref{ex:equator} and \ref{ex: circle}) but not greater (Corollary \ref{cor:per2}) and is monotone in convex order (Corollary \ref{cor:convex-order}).

In the next section, we prove our main theorem regarding $B(\mu)$, while Section~\ref{S:further-properties} is reserved for the development of various properties of the mapping $B :P(X) \rightarrow P(X)$.

\section{Wasserstein regularized Riemannian barycenter $B(\mu)$: existence, uniqueness and characterization.}
In this section, we prove that the Wasserstein regularized barycenter $B(\mu)$ of $\mu \in P(X)$ is well defined; that is, the weak-* limit $B(\mu)$ of the $\mu_\epsilon$ as $\epsilon \rightarrow 0$ exists.  Moreover, we characterize $B(\mu)$ as the probability measure supported on the set  $b(\mu)$ of barycentric points of $\mu$ which is closest to the reference measure ${\bf m}$ in the Wasserstein distance.

\begin{definition}[Regularity of $X$]\label{def:regularity}
	We say that the metric measure space $(X,d,{\bf m })$ is \emph{regular} if for every $\nu \in P(M)$, any minimizer $\gamma$ in the definition \eqref{eqn: wass distance} of $W_2({\bf m}, \nu)$ is concentrated on the graph of a function; that is, $\gamma =(Id,T)_\# \bf m$ for a map $T:X \rightarrow X$.
\end{definition}	
As mentioned in the introduction, many metric measure spaces are regular, including any compact smooth Riemannian manifold where the references measure $\bf m$ is absolutely continuous with respect to volume, and other more singular spaces such as  $CD(K,N)$ spaces. The case where ${\bf m}$ is (normalized) Riemannian volume is our motivating example.

It is well known that regularity ensures that the function $\nu \mapsto \epsilon W_2^2(\nu, {\bf m})$ is strictly convex with respect to linear interpolation between measures in $P(X)$ (see, e.g. \cite[Proposition 7.19]{Santambrogio15p}).  Since the functional $\nu \mapsto \int_X \int_X d^2(y,x)d\mu(x)d\nu(y)$ is linear, we see that $F_\epsilon$ in \eqref{eq:F-epsilon} is strictly convex in $\nu$ (with respect to linear interpolation); therefore its minimizer, $\mu_\epsilon$ is unique. Then, from the weak-* compactness of probability measures 
as $\epsilon \to 0$, a limit point of the $\mu_\epsilon$ exists. The following result establishes uniqueness of this limit point (that is,  the limit exists) and a characterization of it.
\begin{theorem}[Characterization of $B(\mu)$]\label{thm:B-mu-char}
As $\epsilon \rightarrow 0$, $\mu_\epsilon$ converges weak-* to a unique  limit $B(\mu)$ and 
 we have
\begin{align*}
\supp({\bf m}) \cap b(\mu) \subseteq  \supp B(\mu) \subseteq b(\mu), 
\end{align*}
where $b(\mu)$ is the set of barycentric points defined in \eqref{eq:barycenterset}. Moreover, $B(\mu)$ is the unique minimizer of the functional $\nu \mapsto W_2(\nu, {\bf m})$ among $\nu \in P(X)$ with $\supp \nu \subset b(\mu)$, i.e. 
\begin{align}\label{eq:min-in-supp}
 \{B(\mu)\} = \argmin_{\supp \nu \subset b(\mu)} W_2 (\nu, {\bf m}).
\end{align}
\end{theorem}
\begin{proof}
For notational simplicity let us denote 
\begin{align*}
 d_0 = \min \Big( y \mapsto \int _X d^2(x,y) d\mu(x)\Big ). \quad \end{align*}

\paragraph*{Step 1:}
A standard continuity-compactness argument yields the existence of a minimizer $\bar \mu$ in \eqref{eq:min-in-supp}. Now note that the set of probability measures supported on $b(\mu)$ is convex (with respect to linear interpolation) and so as the functional in \eqref{eq:min-in-supp} is strictly convex by \cite[Proposition 7.19]{Santambrogio15p}, the minimizer $\bar \mu$ is unique.
\paragraph*{Step 2:}
Now, let $\mu_0$ be any limit point of  the minimizers $\mu_{\epsilon}$ of $F_{\epsilon}$ as $\epsilon \rightarrow 0$.  We will show that  $\supp \mu_0 \subset b(\mu)$.
Suppose $z \in \supp  \mu_0$, but $z \not\in b(\mu)$. As the set $b(\mu)$ is closed,  there exists $r >0$ such that for the metric ball $B_r(z)$ of radius $r$, $B_{2r} (z) \cap b(\mu) =\emptyset$, and, moreover, $ \mu_0 (B_r(z))  >0$. Since $B_{2r}(z)$ is disjoint from $b(\mu)$, we have $\eta >0$ such that 
\begin{align*}
\int_X d^2 (x,y) d\mu(x) > d_0 + \eta \quad \hbox{for all $y \in B_r(z)$}. 
\end{align*}
and so
\begin{align*}
\int_X \int_X d^2 (x,y) d\mu(x)d\mu_0(y) &\geq  d_0(1-\mu_0 (B_r(z))) + (d_0+\eta)\mu_0 (B_r(z))\\
&=d_0+\eta\mu_0 (B_r(z)).
\end{align*}
Fix a point $\bar x \in b(\mu)$ and define $\tilde \mu =\delta_{\bar x}$. 
Then,
\begin{align*}
& F_\epsilon (\mu_\epsilon)
\le  F_\epsilon (\tilde \mu) =d_0  +\epsilon W_2^2({\bf m} ,\tilde \mu)\\
&\leq \int_X \int_X d^2(x, y)d\mu(x) d\mu_0(y) - \eta  \mu_0 ( B_r(z)) + \epsilon  W_2^2 ({\bf m},  \mu_0)\\
& \quad \quad  +  \epsilon  \left( W_2^2 ({\bf m}, \tilde \mu) -W_2^2 ({\bf m}, \mu_0) \right)
\\
& \le F_\epsilon (\mu_0) - \frac{\eta}{2}  \mu_0 (B_r(z))\quad \hbox{ for sufficiently small $\epsilon$. }
\end{align*}
As $F_{\epsilon}(\nu)$ is continuous in both $\nu$ (with respect to the weak-* topology) and $\epsilon$, and $\mu_0$ is a limit point of the $\mu_\epsilon$ converges, this is a contradiction for small $\epsilon$, establishing $\supp  \mu_0 \subset b(\mu)$.

\paragraph*{Step 3:}
We now show that any limit point $\mu_0$ of the $\mu_\epsilon$ must coincide with $\bar \mu$; this will show that the limit $B(\mu)$ is well defined and equal to $\bar \mu$.

To see this, suppose by contradiction that  there is a limit point  $\mu_0 \neq \bar \mu$ of the $\mu_\epsilon$.  
Then by steps 1 and 2, there is a $\delta >0$ with
\begin{align*}
W_2^2 (\bar \mu, {\bf m}) \le W_2^2 (\mu_0, {\bf m}) - \delta. 
\end{align*}
Choose $\epsilon >0$ sufficiently small so that $W_2(\mu_0, \mu_\epsilon) \le \frac{\delta}{4D}$, where $D$ is the diameter of $X$. We have that 
\begin{eqnarray*}
W^2_2( {\bf m}, \mu_0) -W_2^2({\bf m}, \mu_\epsilon)& =&[W_2( {\bf m}, \mu_0) +W_2({\bf m}, \mu_\epsilon)][W_2( {\bf m}, \mu_0) -W_2({\bf m}, \mu_\epsilon)]\\
&\leq & 2DW_2(\mu_0, \mu_\epsilon) \leq \frac{\delta}{2},
\end{eqnarray*}
and so
\begin{eqnarray*}
F_\epsilon (\mu_\epsilon) & = &\int_X \int_X d^2 (x, y) d\mu(x) d\mu_\epsilon (y) + \epsilon W_2^2 (\mu_\epsilon, {\bf m})\\
& \geq & \int_X \int_X d^2 (x, y) d\mu(x) d\bar \mu (y) + \epsilon [W_2^2 ( \mu_0,{\bf m}) - \frac{\delta}{2}] \\ 
&\geq & \int_X \int_X d^2 (x, y) d\mu(x) d\bar \mu (y) +\epsilon[W_2^2 ( \bar \mu, {\bf m}) +\delta -\frac{\delta}{2}]\\
& >& F_\epsilon (\bar \mu) + \delta/2,
\end{eqnarray*}
which contradicts that $\mu_\epsilon$ is a minimizer of $F_\epsilon$, and therefore establishes that $\mu_\epsilon$ converges to $B(\mu)=\bar \mu$ as $\epsilon \rightarrow 0$.

\paragraph*{Step 4.}
Finally, we verify that $b(\mu) \cap \supp({\bf m}) \subset \supp \bar \mu$. 
Suppose not; then there exists a barycenter point $\bar x \in b(\mu)$ and $r>0$ such that the metric ball $B_{r}(\bar x)$ satisfies ${\bf m}(B_{r}(\bar x))>0$ and $\bar \mu (B_{2r}(\bar x)) =0$.  Now, let $\gamma \in \Gamma ({\bf m}, \bar \mu)
$ be an optimal transport plan from ${\bf m}$ to $\bar \mu$. 
Define $\tilde \mu$ as 
\begin{align*}
\tilde \mu = \bar \mu - \pi^2_\#\gamma_{B_r(\bar x) \times X} + {\bf m}(B_r(\bar x))\delta_{\bar x}
\end{align*}
where  $\pi^2$ is the projection $X\times X  \ni (x, y) \mapsto y \in X$ and $\gamma_{B_r(\bar x) \times X}$ is the restriction of $\gamma$ to the set $B_r(\bar x) \times X$.  Then, as $\bar \mu$ is supported on $b(\mu)$, so is $\tilde \mu$.
Define $\tilde \gamma$, a transport plan from ${\bf m}$ to $\tilde \mu$, as
\begin{align*}
\tilde \gamma = \gamma - \gamma_{B_r(\bar x)\times X} + {\bf m}_{B_r(\bar x)} \otimes \delta_{\bar x}.
\end{align*}
Note this plan $\tilde \gamma$ modifies $\gamma$ by transporting the mass on $B_r(\bar x)$ to the Dirac at $\bar x$. 
From the assumption $\bar \mu (B_{2r}(\bar x) =0$, we see that 
\begin{align*}
\int_{B_r(\bar x) \times X}  d^2 (x, y) d\gamma(x, y) >  4 {\bf m} (B_{r}(\bar x)) \,r^2.
\end{align*}
Therefore, from the obvious inequality 
\begin{align*}
\int  d^2 (x, y) d{\bf m}_{B_r(x)} \otimes \delta_{\bar x} (x, y) \le r^2 {\bf m}(B_r(\bar x)), 
\end{align*}
we get 
\begin{align*}
W_2^2 ({\bf m}, \tilde \mu)  \le  W_2^2 ({\bf m}, \bar \mu) - 3 {\bf m} (B_{r}(\bar x)) \,r^2, 
\end{align*}
contradicting the characterization of $\bar \mu$ as the minimizer of $\nu \mapsto W_2^2({\bf m}, \nu)$ among probability measures supported on $b(\mu)$.
\end{proof}

\begin{remark}[Instability]\label{rem: instability}
 The mapping $\mu \mapsto B(\mu)$ is highly unstable. To see this, consider uniform measure $\mu$ on the round sphere $X$ with ${\bf m} = \vol$. Symmetry considerations easily imply that $B(\mu) = \mu$.   Now, set $\mu^\epsilon:=\epsilon p_n + (1-\epsilon)\mu$, where $p_n$ is the north pole.  Then for any $\epsilon >0$, it is easy to see that $B(\mu^\epsilon) =\delta_{p_n}$.  As $\mu^\epsilon$ is close to $\mu$, but $B(\mu^\epsilon) = p_n$ is far from $B(\mu) =\mu$ in any reasonable topology, we conclude that $B$ is not stable.
  \end{remark}

\section{Further properties of  $B(\mu)$.}\label{S:further-properties}
We now 
investigate some properties of $B(\mu)$. 
First, 
as an immediate corollary of the characterization of $B(\mu)$ in Theorem~\ref{thm:B-mu-char}, we have the following result: 
\begin{corollary}[Support determines $B(\mu)$]\label{cor:support}
If $\supp B(\mu) =\supp B(\nu)$, then $B( \mu )= B( \nu)$.
\end{corollary}

Nest, note that we can interpret  \eqref{eq:min-in-supp} as characterizing $B(\mu)$ as the projection of the reference measure ${\bf m}$ to the set of measures supported on $b(\mu)$, with respect to Wasserstein distance. In the Riemannian setting, the following result makes this characterization more explicit.

\begin{corollary}\label{cor:Riemann-proj}
	Assume $X$ is Riemannian and smooth.  For almost all $x$ (with respect to $\vol$), there is a unique $y \in \argmin_{y \in b(\mu)} d^2(x,y)$.  Denoting the unique minimizer $y=:T(x)$,  and assuming ${ \bf m}$ is absolutely continuous with respect to volume, this $T$ is the optimal transport mapping from ${\bf m}$ to $B(\mu)$; in particular, 
		$$
	T_{\#}{\bf m}=B(\mu).
	$$ 
\end{corollary}

\begin{proof}
Set $f(x) = \min_{y \in b(\mu)}d^2(x,y)$.  By a now standard argument of McCann \cite{m3}, $f$ is Lipschitz and hence differentiable almost everywhere (with respect to $\vol$) by Rademacher's theorem.  Another argument in \cite{m3} implies that $x \mapsto d^2(x,y)$ is differentiable whenever $f$ is, and, as $f(x) -d^2(x,y) \leq 0$ for all $y$, with equality for $y \in \argmin_{y \in b(\mu)} d^2(x,y)$, we have, for all $x$ at which $f$ is differentiable and all $y \in \argmin_{y \in b(\mu)} d^2(x,y)$ 
$$
\nabla f(x) =\nabla_x(d^2(x,y)).
$$
Equivalently, $y =\exp_x(2\nabla f(x)):=T(x)$; that is, $y$ is uniquely determined by $x$.  This holds wherever $f$ is differentiable, and therefore ${\bf m}$-a.e.

Now, for any other $\nu$ supported on $b(\mu)$ letting $T_\nu$ be the optimal map from ${\bf m}$ to $\nu$, we have $d(x,T_(x)) \leq d^2(x,T_\nu(x))$ for almost all $x$ and so
$$
W_2({\bf m}, T_{\#}{\bf m}) \leq \int_{X}d^2(x,T(x))d{\bf m}(x)\leq \int_{X}d^2(x,T_\nu(x))d{\bf m}(x) =W_2^2({\bf m}, \nu)
$$
which establishes minimality of $T_{\#}{\bf m}$ and therefore that $T_{\#}{\bf m} =B(\mu)$, by Theorem \ref{thm:B-mu-char}, and that $T$ is the optimal map between ${\bf m}$ and $B(\mu)$.
\end{proof}

Next, we note that, although $B(\mu)$ may not be supported on a single point, it is at least no more spread out than $\mu$, in the sense that it has lower \textit{variance}, ${\rm var}(\mu):=\min_{y \in X}\int_Xd^2(x,y)d\mu(x)$.  
\begin{corollary}[Variance reduction] \label{cor: variance reduction}
For $\mu \in P(X)$, 
\begin{align*}
 {\rm var} (B( \mu)) \le {\rm var} (\mu).
\end{align*}
Moreover, the equality  holds if and  only if $\supp \mu \subset \supp B(B(\mu))$.  
\end{corollary}
\begin{proof}

 Observe that  ${\rm var} (\mu) =\int_X\int_X d^2(x,y) d\mu(x)d B(\mu)(y)$ since from Theorem~\ref{thm:B-mu-char}, $\supp B(\mu) = b(\mu)$ where $b(\mu)$ is the set of barycenter points of $\mu$.    Now note that
\begin{eqnarray*}
{\rm var}(B(\mu)) &= &\min_{x\in X} \int_X d^2(x,y)d B( \mu)(y)\\
&=& \min_{\nu \in P(X)} \int_X \int_X d^2(x,y)d\nu(x) d B(\mu)(y)\\
&\leq&   \int_X \int_X d^2(x,y)d\mu(x) d B(\mu)(y)\\
&=& {\rm var}(\mu).
\end{eqnarray*}
The equality holds if and only if $\mu $ is a minimizer of $\nu \mapsto \int_X\int_X d^2(x,y) d B(\mu)(x)d\nu(y)$.  This  is equivalent to  $\supp \mu \subset \supp B(B(\mu))$. \end{proof}

The equality case above is illustrated in the following simple examples:
\begin{example}\label{ex:equator}
 Let $X$ be the $n$-dimensional Riemannian round sphere with ${\bf m} = \vol$, and let  $\mu =\frac{1}{2} \delta_{p_s} + \frac{1}{2} \delta_{p_n}$, where $p_s$ and $p_n$ are the south and north poles, respectively. Then, $B(\mu)$ is the uniform probability measure on the equator, and $B(B(\mu)) = \mu$. 
\end{example}

\begin{example}\label{ex: circle}
	Let $X=S_1$ be the circle, $\bf m$ be the normalized arc-length and $\mu =\sum_{i=1}^N\frac{1}{N}\delta_{x_i}$, where $\{x_1,...,x_N\}$ are evenly spaced points on $X$.  The set $b(\mu)$ of minimal points of the function $y \mapsto \sum_{i=1}^N\frac{1}{N} d^2(y,x_i)$ depends on the parity of $N$:
	\begin{enumerate}
		\item If $N$ is odd, the function is minimized at each $x_i$, and so by rotational symmetry the regularized barycenter is $B(\mu) =\sum_{i=1}^N\frac{1}{N}\delta_{x_i} =\mu$; that is, $\mu$ is a fixed point of $B$.
		\item If $N$ is even, the minimizing points are exactly those points $y_i$, $i=1,2,...N$ which are halfway in between two neighbouring $x_i's$, and so $B(\mu) =\sum_{i=1}^N\frac{1}{N}\delta_{y_i}$.  An identical argument then yields $B(B(\mu)) =\mu$.
	\end{enumerate}
	
\end{example}

As we see from the past two examples, it is possible that the operation $\mu \mapsto B(\mu)$ in $P(X)$ may have a periodic orbit. We next prove that no orbit can be periodic with period greater than two.

\begin{corollary}[Period is at most $2$.]\label{cor:per2}
Suppose that $B^N(\mu) =\mu$ for some positive integer $N$.  Then $B^2(\mu) =\mu$.
\end{corollary}
 \begin{proof}
The general inequality ${\rm var} (B(\mu)) \leq {\rm var}(\mu) $ combined with periodicity easily implies that  ${\rm var} (B^k(\mu)) = {\rm var} (\mu) $ for all positive integers $k$. 
In particular, ${\rm var} (\mu) ={\rm var} (B(\mu))$, and Corollary \ref{cor: variance reduction} implies that $\supp(\mu) \subseteq \supp (B^2(\mu)) \subseteq\supp (B^4(\mu)) \subseteq...\subseteq \supp(B^{2N}(\mu)) =\supp(\mu)$.  As the first and last terms in this string of inclusions coincide, we must have equality throughout; in particular, $\supp(\mu) = \supp (B^2(\mu))$.
Therefore,   as $\mu=B^N(\mu)=B(B^{N-1}(\mu))$, we have
$$
\supp(B(B(\mu)))=\supp(B^2(\mu)) =\supp (\mu) =\supp(B(B^{N-1}(\mu)))
$$
and so Corollary \ref{cor:support} implies that $B(\mu) = B^{N-1}(\mu)$.  Applying $B$ to both sides we arrive at
$$ 
 \mu=B^N(\mu) =B(B^{N-1}(\mu))=B(B(\mu)) =B^2(\mu).
$$ 
This completes the proof.
\end{proof}

Finally, we record a version of Jensen's inequality for the Wasserstein regularized barycenter, which might alternatively be interpreted as expressing a convex order between $B(\mu)$ and $\mu$ (see Remark \ref{rem: convex order} below).
  Recall that a function $\phi: X\to \R$ is said to be geodesically convex if for each geodesic segment $\sigma : [0, 1]\to X$, the function $\phi (\sigma(t))$ is convex. 
\begin{corollary}[Monotonicity in convex order]\label{cor:convex-order}
Assume that $X$ is a Riemannian manifold and $\mu$ is absolutely continuous with respect to $\vol$. For any geodesically convex function $\phi$ on $X$, we have 
$$
\int_{X}\phi(x)d(B(\mu))(x) \leq \int_{X}\phi(x)d\mu(x).
$$
\end{corollary}
\begin{proof}
First, we recall that the classical Jensen's inequality extends to Riemannian manifolds (see, for instance,  \cite[Proposition 2]{EmeryMokobodzki1991}), asserting that, for any $y \in b(\mu)$, and geodesically convex function $\phi$, 
\begin{align}\label{eq:Jensen}
\phi(y) \leq \int_{X}\phi(x)d\mu(x).
\end{align}
Now, as $\supp(B(\mu)) \subseteq b(\mu)$, integrating \eqref{eq:Jensen} yields the desired result.
\end{proof}
We conclude the paper with a brief remark, offering some perspective on the preceding corollary.
\begin{remark}[Martingales and convex order on Riemannian manifolds]\label{rem: convex order}
Recall that a coupling $\pi$ between two probability measures $\mu$ and $\nu$ on $\mathbb{R}^n$ is a (discrete) martingale for $(\mu,\nu)$ if for $\mu$ almost every $x$,
$$
x = \int_{\mathbb{R}^n}yd\pi_x(y)
$$ 
where $\pi_x$ represents the disintegration of $d\pi(x,y) =d\pi_x(y)d\mu(x)$ with respect to $\mu$.  Strassen's coupling theorem \cite{strassen1965existence} asserts that there exists a martingale coupling of $\mu$ and $\nu$ if and only if $\nu$ dominates $\mu$ in convex order; that is $\int_{\mathbb{R}^n}\phi(x)d\mu(x) \leq \int_{\mathbb{R}^n}\phi(x)d\nu(x)$ for all convex $\phi: \mathbb{R}^n \rightarrow \mathbb R$.

\noindent On a metric space $X$, it is natural to define a martingale coupling of $\mu$, $\nu \in P(X)$ to be a coupling $d\pi(x,y) =d\pi_x(y)d\mu(x)$ such for $\mu$-a.e. $x$, we have
$$
x \in b(\pi_x).
$$
By analog with the Euclidean case, we will say that $\mu$ dominates $\nu$ in (geodesically) convex order if we have $\int_{X}\phi(x)d\mu(x) \leq \int_{X}\phi(x)d\nu(x)$ for all geodesically convex $\phi$.  

\noindent On a smooth Riemannian manifold, it is not hard to see (using \eqref{eq:Jensen}) that, if there exists a martingale coupling of $\nu$ and $\mu$ then $\mu$ dominates $\nu$ in convex order; that is, one implication of Strassen's theorem extends to manifolds.  The converse fails in general; on the sphere, for example, the only geodesically convex functions are constants, so \emph{any} $\mu$ dominates any $\nu$ in convex order.  

\noindent The preceding proposition asserts that on any manifold,   $\mu$ dominates $B(\mu)$ in convex order.  In this case, it is worth noting that product measure is a martingale between them. So the collection of pairs $\{(B(\mu), \mu)\}_{\mu \in P(X)}$ is a collection of marginals for which the conclusion of Strassen's theorem extends to manifolds.  We also note that when $B^2(\mu)=\mu$,  product measure  is a martingale with respect to either order; that is, it is a martingale for $(B(\mu),\mu)$ \emph{and} for $(\mu,B(\mu)) = (B(B(\mu)),B(\mu))$.
\end{remark}

\bibliographystyle{plain}
 \bibliography{biblio-barycenter}

\begin{thebibliography}{10}

\bibitem{ac}
M.~Agueh and G.~Carlier.
\newblock {Barycenters in the Wasserstein space}.
\newblock {\em SIAM J. Math. Anal.}, 43(2):904--924, 2011.

\bibitem{BigotKlein2013}
J.~Bigot and T.~Klein.
\newblock Consistent estimation of a population barycenter in the {W}asserstein
  space.
\newblock {\em Proceedings of the International Conference Statistics and its
  Interaction with Other Disciplines}, pages 153--157, 2013.

\bibitem{bren}
Y.~Brenier.
\newblock Decomposition polaire et rearrangement monotone des champs de
  vecteurs.
\newblock {\em C.R. Acad. Sci. Pair. Ser. I Math.}, 305:805--808, 1987.

\bibitem{Cavalletti20151367}
F.~Cavalletti and M.~Huesmann.
\newblock Existence and uniqueness of optimal transport maps.
\newblock {\em Annales de l'Institut Henri Poincare (C) Non Linear Analysis},
  32(6):1367 -- 1377, 2015.

\bibitem{EmeryMokobodzki1991}
M.~{\'E}mery and G.~Mokobodzki.
\newblock Sur le barycentre d'une probabilit\'e dans une vari\'et\'e.
\newblock In {\em S\'eminaire de {P}robabilit\'es, {XXV}}, volume 1485 of {\em
  Lecture Notes in Math.}, pages 220--233. Springer, Berlin, 1991.

\bibitem{Gigli2012}
N.~Gigli.
\newblock Optimal maps in non branching spaces with ricci curvature bounded
  from below.
\newblock {\em Geometric and Functional Analysis}, 22(4):990--999, 2012.

\bibitem{Gigli2016}
N.~Gigli, T.~Rajala, and K.-T. Sturm.
\newblock Optimal maps and exponentiation on finite-dimensional spaces with
  ricci curvature bounded from below.
\newblock {\em The Journal of Geometric Analysis}, 26(4):2914--2929, 2016.

\bibitem{KimPass2017}
Y.-H. Kim and B.~Pass.
\newblock Wasserstein barycenters over riemannian manifolds.
\newblock {\em Advances in Mathematics}, 307:640 -- 683, 2017.

\bibitem{Lott-Villani}
J.~Lott and C.~Villani.
\newblock Ricci curvature for metric-measure spaces via optimal transport.
\newblock {\em Annals of Mathematics}, 169(3):903--991, 2009.

\bibitem{m3}
R.~McCann.
\newblock { Polar factorization of maps on Riemannian manifolds}.
\newblock {\em Geom. Funct. Anal.}, 11:589--608, 2001.

\bibitem{RabinPeyreDelonBernot2012}
J.~Rabin, G.~Peyre, J.~Delon, and M.~Bernot.
\newblock {Wasserstein barycenter and its application to texture mixing}.
\newblock In {\em Scale Space and Variational Methods in Computer Vision},
  pages 435--446. 2012.

\bibitem{Santambrogio15p}
F.~Santambrogio.
\newblock {\em Optimal Transport for Applied Mathematicians: Calculus of
  Variations, PDEs and Modeling}, volume~87 of {\em Progress in nonlinear
  differential equations and their applications}.
\newblock {Birkh\"{a}user}, Heidelberg, 2015.

\bibitem{strassen1965existence}
V.~Strassen.
\newblock The existence of probability measures with given marginals.
\newblock {\em The Annals of Mathematical Statistics}, pages 423--439, 1965.

\bibitem{sturm06}
K.-T. Sturm.
\newblock On the geometry of metric measure spaces. {I}.
\newblock {\em Acta Math.}, 196(1):65--131, 2006.

\bibitem{sturm06a}
K.-T. Sturm.
\newblock On the geometry of metric measure spaces. {II}.
\newblock {\em Acta Math.}, 196(1):133--177, 2006.

\bibitem{V}
C.~Villani.
\newblock {\em Topics in optimal transportation}, volume~58 of {\em Graduate
  Studies in Mathematics.}
\newblock American Mathematical Society, Providence, 2003.

\bibitem{V2}
C.~Villani.
\newblock {\em Optimal transport: old and new}, volume 338 of {\em Grundlehren
  der mathematischen Wissenschaften.}
\newblock Springer, New York, 2009.

\end{thebibliography}

\end{document}